\documentclass[11pt]{article}

\usepackage[english]{babel}
\usepackage{amsmath,amssymb,amsthm}
\usepackage{graphicx,color}

\newtheorem{theorem}{Theorem}[section]
\newtheorem{proposition}[theorem]{Proposition}
\newtheorem{corollary}[theorem]{Corollary}

\theoremstyle{definition}
\newtheorem*{remark}{Remark}

\newcommand{\R}{\mathbb{R}}
\newcommand{\C}{\mathbb{C}}
\renewcommand{\H}{\mathbb{H}}
\newcommand{\Circle}{\mathcal{C}}
\newcommand{\Ball}{\mathcal{B}}
\newcommand{\Halfdisk}{\mathcal{D}}
\newcommand{\Halfcircle}{\mathcal{C}}
\newcommand{\bd}{\partial}

\begin{document}

\title{Intersection probabilities for a chordal SLE path and a semicircle}

\author{
Tom Alberts\\
{\small Courant Institute of Mathematical Sciences, 251 Mercer St.}\\
{\small New York University, New York, NY 10012 USA}\\
{\small \texttt{alberts@courant.nyu.edu}}
\and
Michael J.~Kozdron\footnote{Research supported in part by the Natural Sciences and Engineering Research Council of Canada.}\\
{\small Department of Mathematics \& Statistics, College West 307.14}\\
{\small University of Regina, Regina, SK S4S 0A2 Canada}\\
{\small \texttt{kozdron@stat.math.uregina.ca}}
}

\date{March 7, 2008}

\maketitle

\begin{abstract}
We derive a number of estimates for the probability that a chordal
SLE$_\kappa$ path in the upper half plane $\H$ intersects a semicircle
centred on the real line. We prove that if $0<\kappa <8$ and $\gamma:[0,\infty) \to \overline{\H}$ is a chordal
SLE$_\kappa$ in $\H$ from $0$ to $\infty$, then $P\{\gamma[0,\infty)
\cap \Halfcircle(x;rx) \neq \emptyset\} \asymp r^{4a-1}$ where
$a=2/\kappa$ and $\Halfcircle(x;rx)$ denotes the semicircle centred
at $x>0$ of radius $rx$, $0<r\le 1/3$, in the upper half plane. As
an application of our results, for $0<\kappa<8$, we derive an
estimate for the diameter of a chordal SLE$_\kappa$ path in $\H$
between two real boundary points 0 and $x>0$.  For $4<\kappa<8$, we
also estimate  the probability that an entire semicircle on the real line is swallowed at once by a  chordal SLE$_\kappa$ path in $\H$ from $0$ to $\infty$.
\end{abstract}



\section{Introduction}

The primary purpose of this paper is to derive estimates for the probability that
a chordal SLE path in $\mathbb{H}$ from $0$ to $\infty$ intersects a
semicircle in the upper half plane centred at a fixed point $x>0$ on the real line. Specifically, suppose that  $\gamma:[0,\infty) \to \overline{\H}$ is a chordal SLE$_\kappa$ in $\H$ from $0$ to $\infty$ where $\kappa \in (0,8)$.  For $\epsilon>0$ and $x\in\R$,
denote the semicircle of radius $\epsilon$ centred at $x$ in the upper half plane by $\Halfcircle(x; \epsilon)$.  In the  $0 < \kappa < 8$ regime, we will derive estimates for the intersection probability
$$P\{\gamma[0,\infty) \cap \Halfcircle(x; rx) \neq \emptyset\}$$
where  $0 < r \leq 1/3$ and $x>0$.
An example is shown in Figure~\ref{figure1.new}.

We conclude the introduction with the statement of our primary theorems. Most of this paper is devoted to their proof. In the final sections we give some applications of our results. Recall that $g(r) \asymp h(r)$ if there exist non-zero, finite constants $c_1$ and $c_2$ such that $c_1 h(r) \le g(r) \le c_2 h(r)$. Furthermore, $g(r) \sim h(r)$ if $g(r)/h(r) \to 1$ as $r \downarrow 0$.

\begin{figure}[h]
\begin{center}
\includegraphics[height=1.2in]{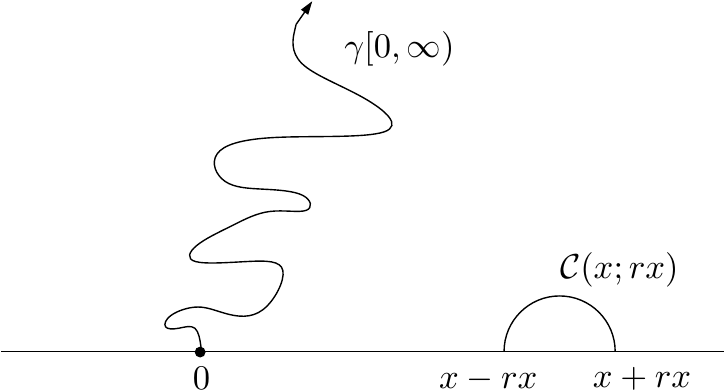}
\caption{The event $\{\gamma[0,\infty) \cap \Halfcircle(x; rx) = \emptyset\}$ in the $0 < \kappa \le 4$ case.}\label{figure1.new}
\end{center}
\vspace{-.25in}
\end{figure}

\begin{theorem}\label{the.theorem}
Suppose  $x>0$ is a real number,  $0< r \leq 1/3$, and
$\Halfcircle(x; rx) = \{x + rxe^{i\theta}: 0 < \theta < \pi\}$
 denotes the semicircle of radius $rx$ centred at $x$ in the upper half plane.
Suppose further that $\gamma:[0,\infty) \to \overline{\H}$ is a chordal SLE$_\kappa$ in $\H$ from $0$ to $\infty$.
\begin{description}
\item[(a)]  If $0<\kappa<8$, then
$P\{\gamma[0,\infty) \cap \Halfcircle(x; rx) \neq \emptyset\} \asymp  r^{\frac{8-\kappa}{\kappa}}$.
\item[(b)] If $\kappa=8/3$, then
$P\{\gamma[0,\infty) \cap \Halfcircle(x; rx) \neq \emptyset\} = 1- (1-r^{2})^{5/8}$.
\end{description}
 \end{theorem}

\begin{remark}
The only reason that we distinguish $\kappa=8/3$ is because an exact computation is possible for this case.  Of course, we can immediately conclude from~(b) that if $\kappa=8/3$, then
$$P\{\gamma[0,\infty) \cap \Halfcircle(x; rx) \neq \emptyset\} \sim \frac{5}{8} r^{2}$$
as $r \downarrow 0$, which is consistent with Theorem~\ref{the.theorem}~(a).
\end{remark}

By an appropriate series of conformal transformations, an equivalent formulation of Theorem~\ref{the.theorem} gives an estimate for the diameter of a chordal SLE$_\kappa$ path in $\H$ from $0$ to $x>0$. This is illustrated in Figure~\ref{figure2.new}.

\begin{figure}[h]
\begin{center}
\includegraphics[height=1.1in]{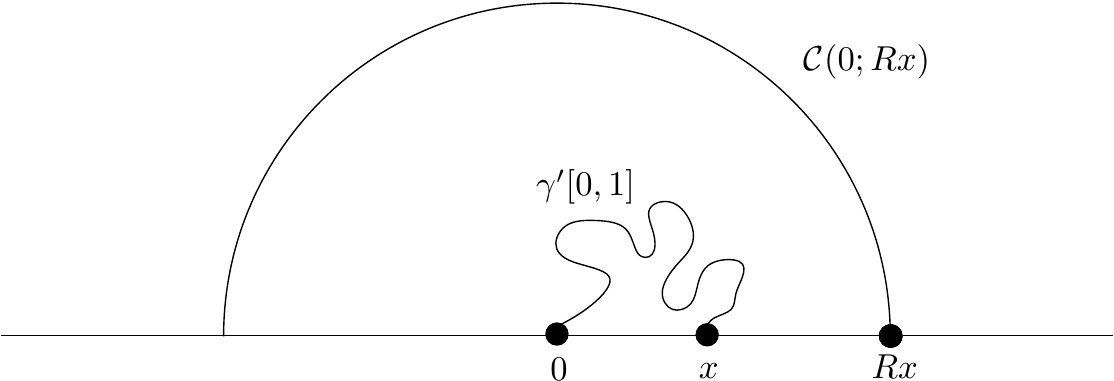}
\caption{The event $\{\gamma'[0,1] \cap \Halfcircle(0;Rx) = \emptyset\}$ in the $0 < \kappa \le 4$ case.}\label{figure2.new}
\end{center}
\end{figure}

\vspace{-.25in}

\begin{corollary}\label{the.theorem.cor}
Suppose $x>0$ is a real number, $R \ge 3$, and $\Halfcircle(0; Rx) = \{Rxe^{i\theta}: 0 < \theta < \pi\}$ denotes the circle of radius $Rx$ centred at $0$ in the upper half plane.
Suppose further that $\gamma':[0,1] \to \overline{\H}$ is a chordal SLE$_\kappa$ in $\H$ from $0$ to $x$.
\begin{description}
\item[(a)]  If $0<\kappa<8$, then
$P\{\gamma'[0,1] \cap \Halfcircle(0; Rx) \neq \emptyset\} \asymp  R^{\frac{\kappa-8}{\kappa}}$.
\item[(b)] If $\kappa=8/3$, then
$P\{\gamma'[0,1] \cap \Halfcircle(0; Rx) \neq \emptyset\} = 1- (1-R^{-2})^{5/8}$.
\end{description}
\end{corollary}

The outline of the remainder of the paper is as follows.  In Section~\ref{Sect.notation}, we introduce some notation. The proof of Theorem~\ref{the.theorem} is then given in Section~\ref{Sect.Thm}. In Section~\ref{Sect.diam} we derive Corollary~\ref{the.theorem.cor} and then conclude in Section~\ref{Sect.general} by using
Theorem~\ref{the.theorem} to derive  two other intersection probabilities for a chordal SLE path and a semicircle centred on the real line.
(These are given by Theorem~\ref{the.theorem.general} and Corollary~\ref{the.corollary.general}.) In particular, for $4 < \kappa <8$, we estimate the probability that an entire semicircle on the real line is swallowed at once by a  chordal SLE$_\kappa$ path in $\H$ from $0$ to $\infty$.

\section{Notation}\label{Sect.notation}

We now introduce the notation that will be used throughout the remainder of the paper.  Let $\C$ denote the set of complex numbers, and write $\H =\{z \in \C: \Im(z)>0\}$ to denote the upper half plane.
If $\epsilon>0$ and $ z\in \C$,
we write
$\Ball(z; \epsilon) = \{ w \in \C : |z-w| < \epsilon\}$
for the ball of radius $\epsilon$ centred at $z$.
If $x \in \R$, then the half disk and semicircle of radius $\epsilon$ centred at $x$ in the upper half plane are given by
\begin{equation}\label{defnDr.feb22}
\Halfdisk(x; \epsilon) =\Ball(x; \epsilon)  \cap \H =\{x +\rho  e^{i\theta}: 0 < \theta < \pi, 0< \rho < \epsilon\}
\end{equation}
and
\begin{equation}\label{defnCr.feb22}
\Halfcircle(x; \epsilon) = \bd \Ball(x; \epsilon) \cap \H = \{x + \epsilon e^{i\theta}: 0 < \theta < \pi\},
\end{equation}
respectively.

The chordal Schramm-Loewner evolution in $\H$ from 0 to $\infty$ with parameter $\kappa = 2/a$ is the solution of the differential equation
\begin{equation}\label{SLEeqn}
\bd_t g_t(z) = \frac{a}{g_t(z)-U_t}, \;\;\; g_0(z)=z,
\end{equation}
where $z \in \H$ and $U_t=-B_t$ is a standard one-dimensional Brownian motion with $B_0=0$.
It is a hard theorem to prove that there exists a curve $\gamma:[0,\infty) \to \overline{\H}$ with $\gamma(0)=0$ which generates the maps $\{g_t, \; t \ge 0\}$.
More precisely, for $z \in\H$,  let $T_z$ denote the first time of explosion of the chordal Loewner equation~(\ref{SLEeqn}), and define the hull $K_t$ by
$K_t  = \overline{\{z \in \H : T_z <t\}}$.
The hulls $\{K_t, \;t\ge 0\}$ are an increasing family of compact sets in $\overline{\H}$ and $g_t$ is a conformal transformation of $\H \setminus K_t$ onto $\H$.  For all $\kappa>0$, there is a continuous curve $\{\gamma(t), \; t \ge 0\}$ with $\gamma:[0,\infty) \to \overline{\H}$ and  $\gamma(0)=0$ such that $\H \setminus K_t$ is the unbounded connected component of $\H \setminus \gamma(0,t]$ (wp1).  The behaviour of the curve $\gamma$ depends on the parameter $\kappa$ (or, equivalently, the value of $a$).  If $a\ge 1/2$ (i.e., $0<\kappa \le 4$), then $\gamma$ is a simple curve with $\gamma(0,\infty) \subset \H$ and $K_t = \gamma(0,t]$. If $1/4 < a< 1/2$ (i.e., $4<\kappa <8$), then $\gamma$ is a non-self-crossing curve with self-intersections and $\gamma(0,\infty) \cap \R \neq \emptyset$; that is, $K_t \neq \gamma(0,t]$.  Although the present work will not be concerned with the case $a \le 1/4$ (i.e., $\kappa \ge 8$), it is worth recalling that for this regime $\gamma$ is a space-filling, non-self-crossing curve.
Let $\mu^{\#}_{\H}(0,\infty)$ denote the chordal SLE$_{\kappa}$ probability measure on paths in $\H$ from 0 to $\infty$. If $D \subset \C$ is a simply connected domain and $z$, $w$ are distinct points in $\bd D$, then $\mu^{\#}_{D}(z,w)$, the chordal SLE$_{\kappa}$ probability measure on paths in $D$ from $z$ to $w$, is defined to the image of $\mu^{\#}_{\H}(0,\infty)$ under a conformal transformation $f:\H \to D$ with $f(0)=z$ and $f(\infty)=w$. In other words, SLE$_{\kappa}$ in $D$ from $z$ to $w$ is simply the conformal image of
SLE$_{\kappa}$ in $\H$ from $0$ to $\infty$. For further details about SLE, consult~\cite{SLEbook}.

\section{Proof of Theorem~\ref{the.theorem}}\label{Sect.Thm}

In this section we prove Theorem~\ref{the.theorem}. The proof is divided into four subsections; the first three subsections are needed to establish~(a) while the fourth subsection is needed to establish~(b). For the lower bound in both the $0<\kappa \le 4$ and $4 <\kappa <8$ cases, we are able to give an explicit value for the constant. For the upper bound, however, all that can be determined is the existence of a constant. 

\subsection{The upper bound}\label{Sect.upper}

Throughout this section, suppose that $\gamma:[0,\infty) \to \overline \H$ is a chordal  SLE$_\kappa$ in $\H$ from $0$ to $\infty$ with $0<\kappa<8$ and  $a=2/\kappa$.  The primary tool we need to establish the upper bound  in Theorem~\ref{the.theorem}~(a) is originally due to Beffara~\cite[Proposition~2]{beffara}. We briefly recall the statement here and refer the reader to~\cite{beffara} for further details.

\begin{proposition}\label{Vincent.Prop}
If $z \in \H$, $0<\epsilon \le \Im\{z\}/2$, and $\Ball(z;\epsilon) = \{w \in \C: |z-w|<\epsilon\}$ denotes the ball of radius $\epsilon$ centred at $z$, then
$$P\{ \gamma[0,\infty) \cap \Ball(z;\epsilon) \neq \emptyset\} \asymp \left(\frac{\epsilon}{\Im\{z\}}\right)^{1-\frac{1}{4a}} \left(\frac{\Im\{z\}}{|z|}\right)^{4a-1}$$
where the constants implied by $\asymp$ may depend on $a$.
\end{proposition}

We would like to stress that this proposition holds for all $a>1/4$ (equivalently, all $0<\kappa<8$). The following theorem gives a careful statement of the upper bound that we will establish.

\begin{theorem}\label{the.theorem.up}
Let $0 < r \le 1/3$ and $x>0$. If  $\gamma:[0,\infty) \to \overline{\H}$ is a chordal SLE$_\kappa$ in $\H$ from $0$ to $\infty$ 
with $0<\kappa<8$ and $a=2/\kappa$,
then there exists a constant $c_a$ such that
$$P\{\gamma[0,\infty) \cap \Halfcircle(x; rx) \neq \emptyset\} \leq c_a   r^{4a-1}.$$
 \end{theorem}

Our general strategy for~Theorem~\ref{the.theorem.up}
will be to cover the semicircle $\Halfcircle(x; rx)$ with a sequence of balls and then apply Proposition~\ref{Vincent.Prop} to each ball. This is illustrated in Figure~\ref{figure3.new}.
\begin{figure}[h]
\begin{center}
\includegraphics[height=2.7in]{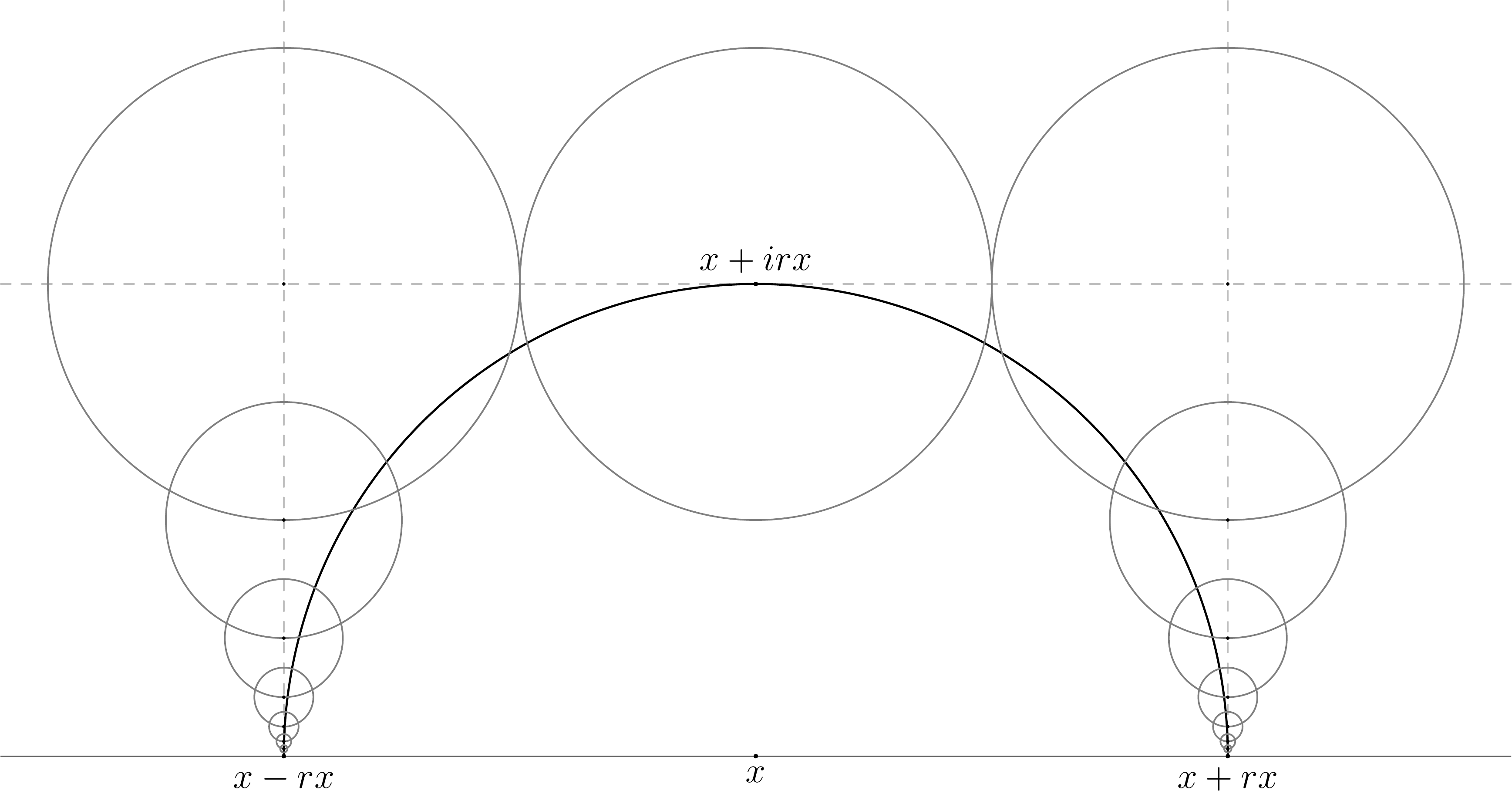}
\caption{The semicircle $\Halfcircle(x; rx)$ covered by a sequence of balls
centred at $\{z_{\pm n}, n=0,1,\ldots\}$. }\label{figure3.new}
\end{center}
\end{figure}

\vspace{-.3in}

\begin{proof}[Proof of Theorem~\ref{the.theorem.up}]
Set $z_0 = x + irx$ and for $n=1,2,\ldots$, let
$z_{ \pm n} =x \pm rx + irx2^{-|n|+1}$. Using Proposition~\ref{Vincent.Prop}, it follows that
$$P\left\{\gamma[0,\infty) \cap\Ball\left(z_{\pm n};\frac{\Im\{z_{\pm n}\}}{2}\right) \neq \emptyset\right\} \asymp
2^{\frac{1}{4a}-1}\left(\frac{\Im\{z_{\pm n}\}} {|z_{\pm n}|}\right)^{4a-1}
 \asymp \frac{r^{4a-1}}{2^{(4a-1)|n|}} $$
since $|z_{\pm n}| \asymp x$ for $0< r \leq 1/3$.
Hence,
\begin{equation}\label{An}
\sum_{n=-\infty}^{\infty} P\left\{\gamma[0,\infty) \cap\Ball\left(z_n;\frac{\Im\{z_n\}}{2}\right) \neq \emptyset\right\} \asymp
r^{4a-1}.
\end{equation}
But if $\gamma[0, \infty)$ intersects $\Halfcircle(x; rx)$, then it also must intersect (at least) one of $\Ball \left( z_{\pm n}; \frac{\Im\{z_{\pm n}\}}{2} \right)$, as is clear from Figure~\ref{figure3.new}. Hence,~(\ref{An}) implies that  there exists a constant $c_a$ such that
$$
P \left \{ \gamma[0, \infty) \cap \Halfcircle(x; rx) \neq \emptyset \right \} \leq \sum_{n=-\infty}^{\infty} P\left\{\gamma[0,\infty) \cap\Ball\left(z_n;\frac{\Im\{z_n\}}{2}\right) \neq \emptyset\right\} \le  c_a r^{4a-1}
$$
and the proof is complete.
\end{proof}

\subsection{The lower bound for $4<\kappa<8$}\label{Sect.lower.48}

In this section we establish the lower bound in Theorem~\ref{the.theorem}~(a) for $\kappa \in (4,8)$.

\begin{theorem}\label{the.theorem.low}
Let $0 < r \le 1/3$ and $x>0$.  If $\gamma:[0,\infty) \to \overline{\H}$ is a chordal SLE$_\kappa$ in $\H$ from $0$ to $\infty$
with $4<\kappa<8$ and $a=2/\kappa$, then there exists a constant $c_a$ such that
$$P\{\gamma[0,\infty) \cap \Halfcircle(x; rx) \neq \emptyset\} \ge c_a   r^{4a-1}.$$
 \end{theorem}

\begin{proof}
It is clear that if $\gamma[0, \infty)$ intersects the interval $[x-rx, x+rx]$ then it also intersects the semicircle $\Halfcircle(x; rx)$, as Figure \ref{figure4.new} shows.
\begin{center}
\begin{figure}[h]
\begin{center}
\includegraphics{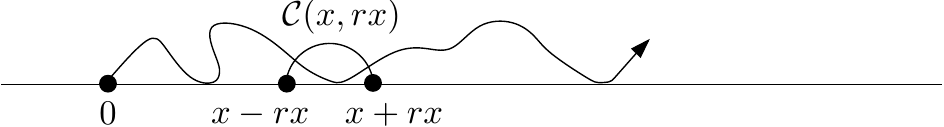}
\caption{The event that $\gamma[0, \infty)$ intersects the interval $[x-rx, x+rx]$.}\label{figure4.new}
\end{center}
\end{figure}
\end{center}
By Proposition~6.34 of \cite{SLEbook} and the scale invariance of SLE,
\begin{align*}
P \left \{ \gamma[0, \infty) \cap [x-rx, x+rx] \neq \emptyset \right \} &= \frac{\Gamma(2a)}{\Gamma(1-2a) \Gamma(4a-1)} \int_0^{\frac{2r}{1+r}} \frac{dt}{t^{2-4a}(1-t)^{2a}} \\
&\geq \frac{\Gamma(2a)}{\Gamma(1-2a) \Gamma(4a-1)} \int_0^{\frac{2r}{1+r}} \frac{dt}{t^{2-4a}(1/2)^{2a}} \\
&\geq \frac{\Gamma(2a)2^{2a}}{\Gamma(1-4a) \Gamma(4a)} (2r)^{4a-1}.
\end{align*}
The first and second inequalities use $0 < r \leq 1/3$.
\end{proof}

\subsection{The lower bound for $0<\kappa \le 4$}\label{Sect.lower.04}

In this section we establish the lower bound in Theorem~\ref{the.theorem}~(a) for $\kappa \in (0,4)$.

\begin{theorem}\label{the.theorem.low.low}
Let $0<r<1$ and $x>0$. If $\gamma:[0,\infty) \to \overline{\H}$ is a chordal SLE$_\kappa$ in $\H$ from $0$ to $\infty$
with $0<\kappa\le 4$ and $a=2/\kappa$, then there exists a constant $c_a$ such that
$$P\{\gamma[0,\infty) \cap \Halfcircle(x; rx) \neq \emptyset\} \ge c_a   r^{4a-1}.$$
\end{theorem}

To prove the theorem we recall the probability that a fixed point $z \in \H$ lies to the left of $\gamma[0,\infty)$. The version that we include may be found in Garban and Trujillo Ferreras~\cite{GT} and is equivalent to the one given by Schramm~\cite{Schramm}.

\begin{proposition}\label{Prop_passleft}
Let $z = \rho e^{i\theta} \in \H$, and set $f(z) = P\{\text{$z$ is to the left of $\gamma[0,\infty)$}\}$.  By scaling, the function $f$ only depends on  $\theta$ and is given by
$$f(\theta) = \frac{\int_{0}^{\theta} (\sin \alpha)^{4a-2} \; d\alpha }{\int_{0}^{\pi} (\sin \alpha)^{4a-2} \; d\alpha}.$$
\end{proposition}

\begin{proof}[Proof of Theorem~\ref{the.theorem.low.low}]
Figure~\ref{figure5.new} clearly shows that
$$P\{\gamma[0,\infty) \cap \Halfcircle(x; rx) \neq \emptyset\}
\ge P\{\text{$x+irx$ is to the left of $\gamma[0,\infty)$}\}.$$
\begin{center}
\begin{figure}[h]
\begin{center}
\includegraphics[height=1.5in]{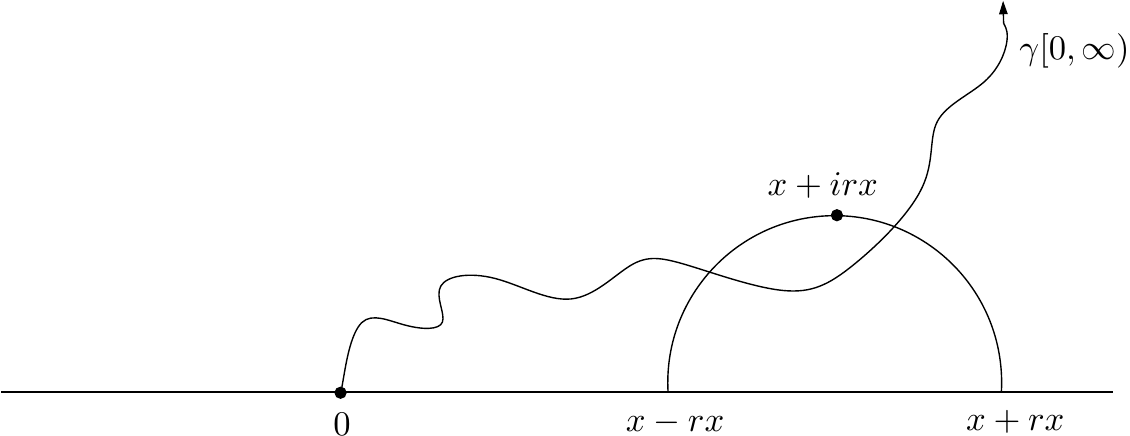}
\caption{The point $z= x+irx$ is to the left of $\gamma[0,\infty)$.}\label{figure5.new}
\end{center}
\end{figure}
\end{center}
Since $\arg(x+irx) = \arctan(r)$ and since
$2\sin t \ge t$
for $0\le t \le \pi/4$, we conclude from Proposition~\ref{Prop_passleft} that
\begin{align}\label{arctan}
P\{\text{$x+irx$ is to the left of $\gamma[0,\infty)$}\} \cdot \int_{0}^{\pi} (\sin \alpha)^{4a-2} \; d\alpha &=\int_{0}^{\arctan(r)} (\sin \alpha)^{4a-2} \; d\alpha\notag\\
&\ge \frac{1}{2}
\int_{0}^{\arctan(r)} \alpha^{4a-2} \; d\alpha\notag\\
&= \frac{\arctan^{4a-1}(r)}{8a-2}.
\end{align}
Since $8\arctan t \ge \pi t$ for $0\le t\le 1$, we see that~(\ref{arctan}) implies that there exists a constant $c_a$, namely
$$c_a = \frac{\pi^{4a-1}}{4^{6a-1}(4a-1) \int_{0}^{\pi} (\sin \alpha)^{4a-2} \; d\alpha},$$
such that
$P\{\text{$x+irx$ is to the left of $\gamma[0,\infty)$}\} \ge c_a r^{4a-1}$.
\end{proof}

 \subsection{The $\kappa=8/3$ case}\label{83section}

In this section we take  $\kappa =8/3$ and complete the proof of Theorem~\ref{the.theorem}~(b). The key fact that is needed is the restriction property of chordal SLE$_{8/3}$. Indeed, the  following remarkable formula due to Lawler, Schramm, and Werner solves the $\kappa=8/3$ case immediately.
See Theorem~6.17 of~\cite{SLEbook} for a proof; compare this with Proposition~9.4 and Example~9.7 of~\cite{SLEbook} as well.

\begin{proposition}\label{83exact}
If $\gamma:[0,\infty) \to \overline{\H}$ is a chordal SLE$_{8/3}$ in $\H$ from 0 to $\infty$, and $A$ is a bounded subset of $\H$ such that $\H \setminus A$ is simply connected, $A = \H \cap \overline{A}$, and $0 \not\in \overline{A}$,
then
$$P\{\gamma[0,\infty) \cap A = \emptyset \} = \left[\Phi'_A(0)\right]^{5/8}$$
where $\Phi_A: \H \setminus A \to \H$ is the unique conformal transformation of $\H\setminus A$ to $\H$ with $\Phi_A(0)=0$ and
$\Phi_A(z) \sim z$ as $z \to \infty$.
\end{proposition}

Applying Proposition~\ref{83exact} to our situation implies that
$$P\{\gamma[0,\infty) \cap \Halfcircle(x; rx) = \emptyset\} =P\{\gamma[0,\infty) \cap \Halfdisk(x; rx) = \emptyset\} = \left[\Phi_{\Halfdisk(x; rx)}'(0)\right]^{5/8}$$
where $\Halfdisk(x; rx)$ is the half disk of radius $rx$ centred at $x$ in the upper half plane as given by~(\ref{defnDr.feb22}) and $\Phi_{\Halfdisk(x; rx)}(z)$ is the conformal transformation from $\H \setminus \Halfdisk(x; rx)$ onto $\H$ with  $\Phi_{\Halfdisk(x; rx)}(0)=0$ and
$\Phi_{\Halfdisk(x; rx)}(z) \sim z$ as $z \to \infty$. In fact,  the exact form of $\Phi_{\Halfdisk(x; rx)}(z)$ is given by
$$\Phi_{\Halfdisk(x; rx)}(z)  = z+ \frac{r^{2}x^2}{z-x}+r^{2}x.$$
Note that $\Phi_{\Halfdisk(x; rx)}(0)=0$, $\Phi_{\Halfdisk(x; rx)}(\infty)=\infty$, and $\Phi'_{\Halfdisk(x; rx)}(\infty)=1$.
We calculate $\Phi_{\Halfdisk(x; rx)}'(0)=1-r^{2}$ and therefore conclude that
$$P\{\gamma[0,\infty) \cap \Halfcircle(x; rx) = \emptyset\} =(1-r^{2})^{5/8}$$
establishing Theorem~\ref{the.theorem}~(b).

\begin{remark}
It is worth noting that Proposition~\ref{83exact} with the exact form of the conformal transformation $\Phi_{\Halfdisk(x; rx)}: \H \setminus \Halfdisk(x; rx)\to \H$ was used by Kennedy~\cite{Kennedy03} to produce strong numerical evidence that the scaling limit of planar self-avoiding walk is chordal SLE$_{8/3}$.
\end{remark}

\section{Estimating the diameter of a chordal SLE path}\label{Sect.diam}

In this section, we derive Corollary~\ref{the.theorem.cor} from Theorem~\ref{the.theorem}.  The proof is not difficult; the basic idea is to determine the appropriate sequence of conformal transformations and use the conformal invariance of chordal SLE. Recall that if $D \subset \C$ is a simply connected domain and $z$, $w$ are two distinct points in $\bd D$, then chordal SLE$_\kappa$ in $D$ from $z$ to $w$ is defined to be the conformal image of chordal SLE$_\kappa$ in $\H$ from $0$ to $\infty$ as discussed in Section~\ref{Sect.notation}.

Let $x>0$ be real, and suppose that $\gamma':[0,1] \to \overline \H$ is an SLE$_\kappa$ in $\H$ from $0$ to $x$. We also note that we are not interested in the parametrization of the SLE path, but only in the set of points visited by its trace.
Suppose that $R\ge 3$, and consider
$\Halfcircle(0;Rx)  = \{Rxe^{i\theta}: 0 < \theta < \pi\}$. For $z \in \H$, let
$$h(z) = \frac{R^2}{R^2-1} \frac{z}{x-z}$$
so that $h:\H \to \H$ is a conformal (M\"obius) transformation with $h(0)=0$ and $h(x)=\infty$. It is straightforward (although a bit tedious) to verify that
$$h \left( \Halfcircle(0;Rx) \right) = \Halfcircle \left( -1; \frac{1}{R} \right).$$

If $\gamma:[0,\infty) \to \overline \H$ is a chordal SLE$_\kappa$ in $\H$ from $0$ to $\infty$, then the conformal invariance of SLE implies that
$$P\{\gamma'[0,1] \cap \Halfcircle(0;Rx)  \neq \emptyset\} =
P\{h(\gamma'[0,1]) \cap h(\Halfcircle(0;Rx))  \neq \emptyset\} =
 P\left\{\gamma[0,\infty) \cap \Halfcircle\left(-1, \frac{1}{R} \right) \neq \emptyset\right\}.$$
By the symmetry of SLE about the imaginary axis,
$$
P\left\{\gamma[0,\infty) \cap \Halfcircle\left(-1, \frac{1}{R} \right) \neq \emptyset\right\}=
P\left\{\gamma[0,\infty) \cap \Halfcircle\left(1, \frac{1}{R} \right) \neq \emptyset\right\} \asymp R^{1-4a},
$$
where the last bound follows from Theorem~\ref{the.theorem} with $r = 1/R$.

\section{An application of Theorem~\ref{the.theorem}}\label{Sect.general}

In this section, we derive estimates for two more intersection probabilities for a chordal SLE path and a semicircle centred on the real line.  In particular,
Corollary~\ref{the.corollary.general} gives an estimate 
in the  $4 < \kappa <8$ regime for the probability that an entire semicircle is swallowed at once by a  chordal SLE$_\kappa$ path in $\H$ from $0$ to $\infty$.

By the scaling properties of SLE, we may rewrite Theorem~\ref{the.theorem} 
in terms of a semicircle centred at $x>0$ of radius $\epsilon$, $0<\epsilon \le x/3$. In this form, it is seen to generalize a result due to Rohde and Schramm~\cite[Lemma~6.6]{RS}.

\begin{corollary}\label{the.corollary}
Let $x>0$ be a fixed real number, and suppose $0<\epsilon\le x/3$. If $\gamma:[0,\infty) \to \overline{\H}$ is a chordal SLE$_{\kappa}$ in $\H$ from $0$ to $\infty$ with $0<\kappa<8$ and $a=2/\kappa$, then
$$P\{\gamma[0,\infty) \cap \Halfcircle(x;\epsilon) \neq \emptyset \} \asymp \left(\frac{\epsilon}{x}\right)^{4a-1}$$
where $\Halfcircle(x;\epsilon)$ is the semicircle of radius $\epsilon$ centred at $x$ in the upper half plane as given by~(\ref{defnCr.feb22}).
\end{corollary}

We conclude with an application of Corollary~\ref{the.corollary} by combining it with a method due to Dub\'edat~\cite{triangles}. For the remainder of the paper, however, suppose that $4<\kappa<8$; as before, let  $a=2/\kappa$.
Suppose that $0<r \le 1/3$ and consider the two semicircles
\begin{equation}\label{circleR}
\Circle_r = \Halfcircle\left(1-r;\frac{r}{2}\right)
 = \left\{z \in \H : \left|z-1+r \right| = \frac{r}{2}\right\}
 \end{equation}
and
\begin{equation}\label{circleRprime}
\Circle'_r = \Halfcircle\left(1-\frac{3r}{4};\frac{3r}{4}\right)
 = \left\{z \in \H : \left|z-1+\frac{3r}{4} \right| = \frac{3r}{4}\right\}
 \end{equation}
as illustrated in Figure~\ref{figure6.balls}.
\begin{center}
\begin{figure}[h]
\begin{center}
\includegraphics[height=1.3in]{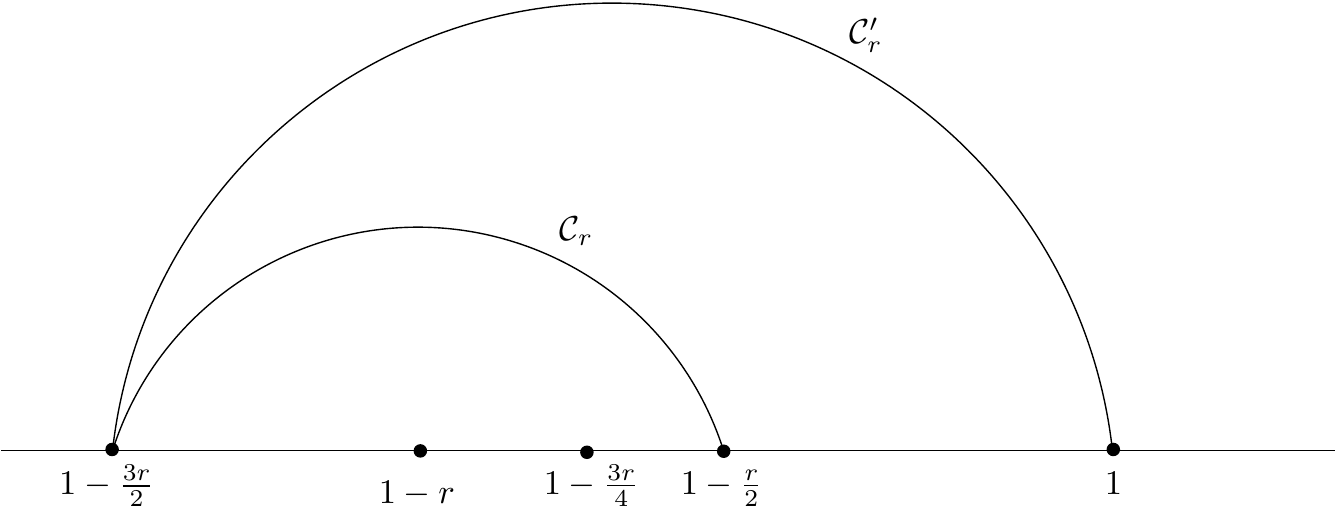}
\end{center}
\caption{The semicircles $\Circle'_r$ and $\Circle_r$.}\label{figure6.balls}
\end{figure}
\end{center}
It follows from Corollary~\ref{the.corollary} that
$P\{\gamma[0,\infty) \cap \Circle'_r \neq \emptyset\} \asymp r^{4a-1}$
and so there exists a constant $c'_a$ such that
$ 1-c'_ar^{4a-1} \le P\{\gamma[0,\infty) \cap \Circle'_r = \emptyset\}$.
However, it clearly follows that
$$P\{\gamma[0,\infty) \cap \Circle'_r = \emptyset\} \le \inf_{z \in \Circle_r} P\{T_z = T_1\}$$
where $T_z$ is the swallowing time of the point $z\in \overline \H$ (and the infimum is over all $z \in \Circle_r$ \emph{not} $z\in \Circle'_r$).
From this we conclude that there exists a constant $c'_a$ such that
\begin{equation}\label{jul18.eq1}
1- c'_a r^{4a-1} \le \inf_{z \in \Circle_r} P\{T_z = T_1\}.
\end{equation}
In order to derive an upper bound for the expression in~(\ref{jul18.eq1}),
we use a method from Dub\'edat~\cite{triangles}. We now outline this method referring the reader to that paper for further details.

Let $g_t$ denote the solution to the chordal Loewner equation~(\ref{SLEeqn}) with driving function $U_t = -B_t$ where $B_t$ is a standard one-dimensional Brownian motion with $B_0=0$. For $t < T_1$, the swallowing time of the point $1$, consider the conformal transformation $\tilde g_t:\H \setminus K_t \to \H$
given by
$$\tilde g_t(z)= \frac{g_t(z) +B_t}{g_t(1)+B_t}, \;\;\; \tilde g_0(z)=z.$$
Note that $\tilde g_t(\gamma(t)) = 0$, $\tilde g_t(1)=1$, $\tilde g_t(\infty)= \infty$, and that $\tilde g_t(z)$ satisfies the stochastic differential equation
$$
d \tilde g_t(z) =
\left[\frac{a}{\tilde g_t(z)} +(1-a)\tilde g_t(z) -1 \right] \frac{dt}{(g_t(1)+B_t)^2}
+ \left[1-\tilde g_t(z) \right] \frac{dB_t}{g_t(1)+B_t}.
$$
If we now perform the time-change
$$\sigma(t) = \int_0^t \frac{ds}{(g_s(1)+B_s)^2},$$
then $\tilde g_{\sigma(t)}(z)$
satisfies the  stochastic differential equation
\begin{equation}\label{gtildesde}
d \tilde g_t(z) =
\left[\frac{a}{\tilde g_t(z)} +(1-a)\tilde g_t(z) -1 \right] dt
+ \left[1-\tilde g_t(z) \right] dB_t
\end{equation}
 For ease of notation, and because it does not concern us at present, we have also denoted the time-changed flow by $\{\tilde g_t(z), \; t\ge 0\}$. Furthermore, it is shown in detail in~\cite{triangles} that for all $\kappa>0$, the time-changed stochastic flow $\{\tilde g_t(z), \; t\ge 0\}$ given by~(\ref{gtildesde}) does not explode in finite time (wp1).

 Therefore, if $F$ is an analytic function on $\H$ such that $\{F(\tilde g_t(z)), \; t \ge 0\}$ is a local martingale, then It\^o's formula (at $t=0$) implies that $F$ must be a solution to the differential equation
\begin{equation}\label{Fode}
w(1-w)F''(w) +[2a-(2-2a)w]F'(w)=0.
\end{equation}
An explicit solution to~(\ref{Fode}) is given by
\begin{equation}\label{nov18eq1}
F(w) =\frac{\Gamma(2a)}{\Gamma(1-2a)\Gamma(4a-1)}  \int_0^w \zeta^{-2a}(1-\zeta)^{4a-2} d\zeta
\end{equation}
which is normalized so that $F(0)=0$ and $F(1)=0$. Note that~(\ref{nov18eq1}) is a
Schwarz-Christoffel transformation of  the upper half plane onto the isosceles triangle whose interior angles are $(1-2a)\pi$, $(1-2a)\pi$, and $(4a-1)\pi$.  The boundary values $F(0)=0$ and $F(1)=1$ imply that two of the vertices of the triangle are at $0$ and $1$, and from~(\ref{nov18eq1}) we conclude that the third vertex of the triangle is at
$$F(\infty) =  \frac{\Gamma(2a) \Gamma(1-2a)}{\Gamma(2-4a)\Gamma(4a-1)}  e^{(1-2a)\pi i}$$
which follows from~(6.2.1) and~(6.2.2) of~\cite{AbSteg}.
Furthermore, using (15.1.7),~(15.1.20),~(6.1.15), and~(6.1.18) of~\cite{AbSteg}, one can show that
$$2\cos((1-2a)\pi) =  \frac{\Gamma(2a) \Gamma(1-2a)}{\Gamma(2-4a)\Gamma(4a-1)}$$
from which it follows that $\Re(F(\infty)) \ge 0$ and that $|F(\infty)-1|=1$ as is to be expected for this isosceles triangle. The image of $\H$ under  $F$ is illustrated in Figure~\ref{figure7.triangle}.
\begin{figure}[h]
\begin{center}
\includegraphics[height=2.2in]{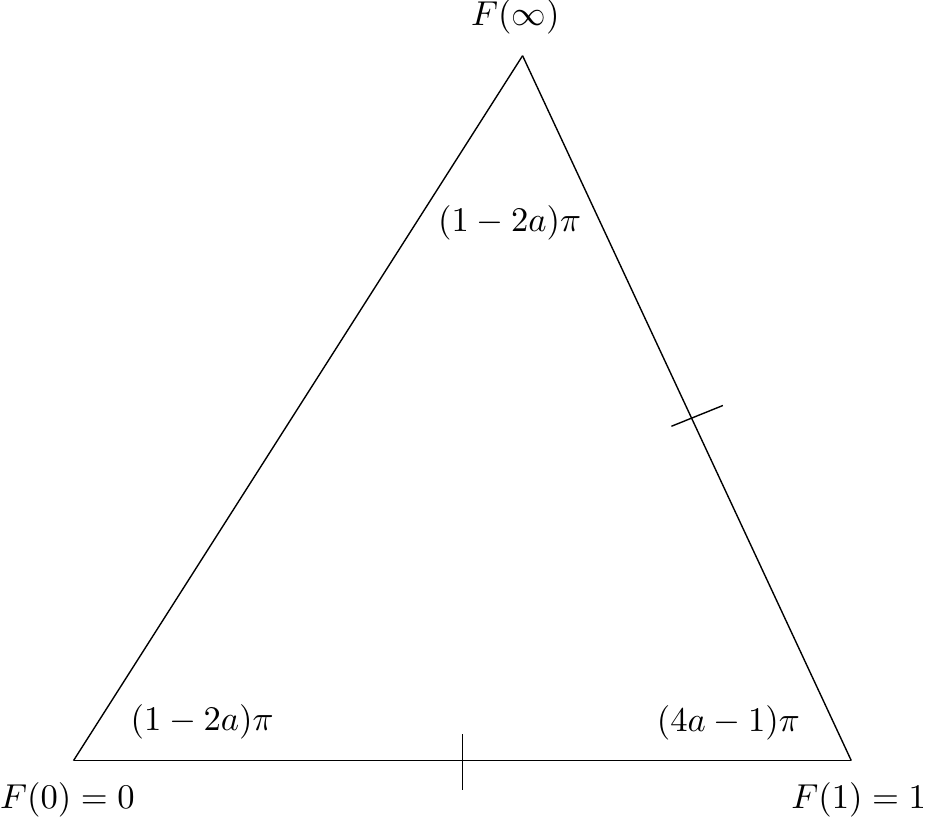}
\caption{The isosceles triangle with vertices at $0$, $1$, and $F(\infty)$.}\label{figure7.triangle}
\end{center}
\end{figure}

We now apply the optional sampling theorem to the martingale $F(\tilde g_{t\wedge T_z \wedge T_1}(z))$ to find (see the discussion surrounding Proposition~1 of~\cite{triangles}) that for $z \in \H$,
\begin{align}\label{jul7.eq1}
F(\tilde g_0(z)) = F(z) &= F(0)P\{T_z < T_1\} +F(1)P\{T_z = T_1\}+F(\infty)P\{T_z > T_1\}\notag\\
&= P\{T_z = T_1\}+F(\infty)P\{T_z > T_1\}.
\end{align}
Consequently, identifying the imaginary and real parts of the previous equation~(\ref{jul7.eq1}) implies that
$$\Re\{F(z)\} =  P\{T_z = T_1\}+\Re\{F(\infty)\}P\{T_z > T_1\}.$$
Since $\Re\{F(\infty)\} \ge 0$, we conclude $P\{T_z = T_1\} \le \Re\{F(z)\} \le |F(z)|$.

But now integrating along the straight line from $0$ to $z$ (i.e., letting $\theta=\arg(z)$, $\zeta = \rho e^{i\theta}$, $0 \le \rho \le |z|$) gives
\begin{align*}
|F(z)|
&= \frac{\Gamma(2a)}{\Gamma(1-2a)\Gamma(4a-1)} \left| \int_0^{|z|} (\rho e^{i\theta})^{-2a}(1-\rho e^{i\theta})^{4a-2} e^{i\theta}d\rho \right|\\
 &\le \frac{\Gamma(2a)}{\Gamma(1-2a)\Gamma(4a-1)}  \int_0^{|z|} \rho^{-2a}|1-\rho|^{4a-2} d\rho\\
&= 1 - \frac{\Gamma(2a)}{\Gamma(1-2a)\Gamma(4a-1)}  \int_{|z|}^1 \rho^{-2a}(1-\rho)^{4a-2}d\rho
\end{align*}
which relied on the fact that $4a-2<0$.

 If $z \in \Circle_r$
so that $0<1-\frac{3r}{2} \le |z| \le 1-\frac{r}{2} <1$ by definition, then
$$\int_{|z|}^1 \rho^{-2a}(1-\rho)^{4a-2}d\rho
\ge \int_{|z|}^1(1-\rho)^{4a-2} d\rho = \frac{(1-|z|)^{4a-1}}{4a-1} \ge \frac{2^{1-4a}}{4a-1}r^{4a-1}.$$
Hence,
$$P\{T_z = T_1\} \le |F(z)| \le 1 - c_a''r^{4a-1}$$
where
$$c''_a= \frac{2^{1-4a}\tilde c_a}{4a-1} \;\;\; \text{and} \;\;\; \tilde c_a=\frac{\Gamma(2a)}{\Gamma(1-2a)\Gamma(4a-1)}.$$
Taking the supremum of the previous expression over all $z \in \Circle_r$ gives us the required upper bound to~(\ref{jul18.eq1}). Hence, we have proved the following theorem.

\begin{theorem}\label{the.theorem.general}
Let $0<r \le 1/3$. If $\gamma:[0,\infty) \to \overline{\H}$ is a chordal SLE$_{\kappa}$ in $\H$ from 0 to $\infty$ with $4 < \kappa < 8$ and $a=2/\kappa$, then there exist constants $c'_a$ and $c''_a$ such that
$$1- c'_a r^{4a-1} \le \inf_{z \in \Circle_r} P\{T_z = T_1\} \le \sup_{z \in \Circle_r} P\{T_z = T_1\} \le   1 - c''_a r^{4a-1}$$
where
$$\Circle_r = \Halfcircle\left(1-r; \frac{r}{2} \right) = \left\{z \in \H : \left|z-1+r \right| = \frac{r}{2}\right\}$$
denotes the circle of radius $r/2$ centred at $1-r$ in the upper half plane as in~(\ref{circleR}).
\end{theorem}

This theorem now yields the following corollary.

\begin{corollary}\label{the.corollary.general}
Let $0<r \le 1/3$. If $\gamma:[0,\infty) \to \overline{\H}$ is a chordal SLE$_{\kappa}$ in $\H$ from 0 to $\infty$ with $4 < \kappa < 8$ and $a=2/\kappa$, then there exist constants $c'_a$ and $c''_a$ such that
$$1- c'_a r^{4a-1} \le
P\{\text{$T_z = T_1$ for all $z \in \Circle_r$} \}
\le   1 - c''_a r^{4a-1}$$
where $\Circle_r$ is given by~(\ref{circleR}) as above.
\end{corollary}

\begin{proof}
Let $z_0=  1-r + \frac{ir}{2}$ so that  $z_0 \in \Circle_r$. Theorem~\ref{the.theorem.general} implies that there exists a constant $c''_a$ such that
\begin{equation}\label{cor2.eq1}
P\{\text{$T_z = T_1$ for all $z \in \Circle_r$} \} \le P\{T_{z_0} = T_1\} \le  \sup_{z \in \Circle_r} P\{T_z = T_1\} \le   1 - c''_a r^{4a-1}.
\end{equation}
As noted earlier, it follows from Corollary~\ref{the.corollary} that
$P\{\gamma[0,\infty) \cap \Circle'_r \neq \emptyset\} \asymp r^{4a-1}$
where $\Circle'_r$ is given by~(\ref{circleRprime}),
and so
there exists a constant $c'_a$ such that
\begin{equation}\label{cor2.eq2}
P\{\text{$T_z = T_1$ for all $z \in \Circle_r$} \} \ge P\{ \gamma[0,\infty) \cap \Circle'_r = \emptyset\} \ge 1- c'_a r^{4a-1}.
\end{equation}
Taking~(\ref{cor2.eq1}) and~(\ref{cor2.eq2}) together completes the proof.
\end{proof}

\section*{Acknowledgements}

This paper had its origins at \emph{Workshop on Random Shapes} held at the Institute for Pure \& Applied Mathematics in March 2007, and was completed during the 2007 \emph{IAS/Park City Mathematics Institute on Statistical Mechanics}. The authors would like to thank both organizations for partial financial support, Peter Jones who organized the IPAM workshop, Scott Sheffield and Tom Spencer who organized the PCMI program, as well as Greg Lawler, Christophe Garban, and Ed Perkins who provided us with a number of valuable comments.

\end{document}